\documentclass{ifacconf}
\usepackage{amsmath,amssymb,amsfonts}
\usepackage{enumitem}
\usepackage{cite}
\usepackage{graphicx}      
\usepackage{natbib}        
\usepackage{float}
\usepackage{xcolor}

\newtheorem{remark}{Remark}
\newtheorem{example}{Example}
\newtheorem{theorem}{Theorem}
\newtheorem{proof}{Proof}

\begin{document}
\begin{frontmatter}

\title{Optimization-based One-side Boundary Control of LWR Traffic Models}

\author[First]{Eryn Vaid} 
\author[Second]{Maria Teresa  Chiri} 
\author[Third]{Roberto Guglielmi}
\author[First]{Gennaro Notomista}

\address[First]{Department of Electrical and Computer
Engineering, University of Waterloo, Waterloo, Canada, (e-mail: evaid@uwaterloo.ca, gennaro.notomista@uwaterloo.ca).}
\address[Second]{Department of Mathematics and Statistics Queen’s University, Kingston, Canada, (e-mail: maria.chiri@queensu.ca)}
\address[Third]{Department of Applied Mathematics, 
   University of Waterloo, Waterloo, Canada, (e-mail: roberto.guglielmi@uwaterloo.ca)}

\begin{abstract}                
In this paper, we study the feasibility of a class of optimization-based boundary control of one-dimensional macroscopic traffic flow models, where stability and invariance are achieved by 
a single boundary 
control. We define the sets of controllers to stabilize the system to a desired state via Lyapunov functionals, and 
to ensure forward invariance of a desired subset via boundary control barrier functionals. The control input is then selected from the intersection of those sets via a convex optimization problem. We determine sufficient conditions to ensure the existence of an optimal boundary control problem achieving both stability and invariance for a generic traffic flux function. Simulation results showcase the behavior of the proposed optimization-based controller applied to conservation laws with several traffic flow functions.

\end{abstract}





\end{frontmatter}

\section{Introduction}
\label{section 1}




Control problems involving Partial Differential Equations (PDEs) appear across many application areas, such as fluid dynamics, heat transfer, structural mechanics, and delay systems. In these settings, two fundamental control objectives are stability, which relates to driving the system toward a desired long-term state, and invariance, which ensures that the system trajectory remains within a specified safe region throughout the time horizon.
Substantial effort has been devoted to the synthesis of stabilizing controllers for PDEs, leading to methodologies based on control Lyapunov functionals and backstepping. The latter has been successfully applied in traffic, thermal, and multi-agent systems, among others, but its practical implementation typically relies on numerical approximations and ad hoc computational procedures (\cite{vazquez2024backstepping}). Invariance has instead been explored primarily through optimization-based controllers, where invariance requirements are encoded as state constraints imposed along the prediction horizon (\cite{troltzsch2010optimal,daudin2023optimal}). Although these methods provide conditions for existence and regularity of the resulting control laws, their computational complexity generally makes them unsuitable for real-time applications—particularly in high-dimensional settings such as fluid dynamics and traffic systems.

Building on these observations, recent work in \cite{chiri2025boundary} introduced a convex optimization framework for PDE boundary control that guarantees both stability and invariance in traffic-flow dynamics. Motivated by the computational limitations of existing optimization-based methods, 
\cite{chiri2025boundary} proposed a boundary-control architecture based on two distinct inputs: one responsible for stabilization and one for enforcing invariance. This two-input framework demonstrated that real-time implementability is achievable when the control design exploits boundary actuation together with appropriately structured convex programs.

In this paper, we investigate whether a single boundary input can simultaneously enforce both objectives. Our contributions are threefold. First, we identify feasibility conditions under which a single boundary controller guaranties stability and invariance. Second, we construct convex single-boundary optimization programs over intervals where the control Lyapunov (\cite{ledyaev1999lyapunov}) and barrier functionals (\cite{ames2019control}) exhibit convexity. Third, we derive structural properties that streamline the associated feasibility analysis. We then validate the proposed methodology through numerical simulations.
Furthermore, while existing works often assume a specific parametrization of the traffic flux function, our approach does not rely on an explicit functional form. Instead, we develop all results under the general structural assumptions satisfied by admissible flux functions in the LWR traffic model (\cite{lighthill1955traffic,richards1956shock}), extending the problem formulation in~\cite{chiri2025boundary}.




The rest of the paper is structured as follows: 
In Section~\ref{section 2} we outline the theory and methodology behind achieving stability via Lyapunov functionals and invariance via control barrier functionals (CBF). In Section~\ref{section 3} we determine the feasibility of a single input controller, for both the left-side and right-side cases. We conclude by validating our results via numerical simulations illustrated in Section~\ref{section 4}.
\section{Optimization-based Traffic Control 
}\label{section 2}
This section provides the essential background for our analysis. We first restate the boundary value problem associated with the conservation law under consideration, then summarize the asymptotic stabilization framework, and finally introduce the barrier functional that formalizes the invariance requirement.

\subsection{Boundary value problem for the LWR Model }\label{subsec:IBVP}

We consider the Lighthill--Whitham--Richards (LWR) model, a macroscopic description of traffic flow based on conservation of vehicles. Let $u(t,x)$ denote the density and {$v(u)$} the mean vehicle speed, so that the traffic flux is 
$f(u)=uv(u)$.
The resulting dynamics are governed by the conservation law
$u_t + f(u)_x = 0$.
Throughout this paper, we assume that the flux $f:[0,u_{\max}] \to \mathbb{R}_+$
is $\mathcal{C}^2$, strictly concave, and satisfies
\(f(0)=f(u_{\max})=0,\)
which ensures the existence of a unique maximizer $\hat{u}\in(0,u_{\max})$.
On an interval $(a,b)$, we consider the initial--boundary value problem (IBVP)
\begin{align}
    &u_t + f(u)_x = 0, &u(0,x)=u_0(x), \label{eq1}\\
    &u(t,a) = \omega_a(t),  &u(t,b)=\omega_b(t), \label{eq2}
\end{align}
with weak entropy solutions and initial data of bounded variation (BV) (\cite{bardos1979first}).

\subsection{Lyapunov-Based Stabilization}

Given an equilibrium density $u^*\in[0, u_{\max}]$, the stabilization problem consists in determining boundary inputs $\omega_a$, $\omega_b$ ensuring Lyapunov stability of~\eqref{eq1}–\eqref{eq2}.  
By following the methods presented in (\cite{7509658,bayen2022control}), a natural candidate Lyapunov functional is
\begin{equation}\label{eq:V-summary}
    V(u(t)) = \frac12\int_a^b (u(t,x)-u^*)^2 \, dx,
\end{equation}
which is well defined and continuous for entropy solutions in BV (\cite{7509658}). 
The derivative of $V$ can be explicitly written in terms of boundary traces and jump contributions (\cite{7509658}). The latter can be removed as they are always nonpositive which decreases the Lyapunov functional, and independent of the boundary inputs $\omega_a$ and $\omega_b$. So, enforcing the stabilization condition \(\dot{V}+\alpha(V(u(t))) \le 0\) for a class \(\mathcal{K}\) function \(\alpha\) reduces to ensuring:
\begin{align}
    (u(t,a)-u^*) f(u(t,a)) - (u(t,b)-u^*) f(u(t,b))\notag\\
    - F(u(t,a)) + F(u(t,b)) \le -\alpha(V(u(t))),\label{eq:boundary-cond-summary}
\end{align}
where $F$ is any primitive of the 
flux~$f$.  
Classical stabilizing boundary controls enforcing \eqref{eq:boundary-cond-summary} typically depend on the boundary traces of the entropy solution, the initial data and are non-local (\cite{7509658}).

\subsection{Barrier-Based Invariance} 

The invariance problem consists in determining boundary inputs \(\omega_a\) and \(\omega_b\) that ensure \eqref{eq1}–\eqref{eq2} admit a solution satisfying a certain safety condition for all times. Assuming such  a condition aims to keep the solution to~\eqref{eq1}–\eqref{eq2}
bounded in a suitable norm, we can express this objective via the following boundary control barrier functional (BCBFal), already used in \cite{chiri2025boundary}:
\begin{equation}
    \label{eq:B-summary}
    B(u(t)) = \bar{u}^2 -\int_a^b u(t,x)^2dx , 
\end{equation}
whose zero superlevel set defines a set of states \(u\) of \(L^2\)-norm bounded by \(\bar{u}\). 
The use of a different norm in the definition of the barrier functional would be more direct, for example imposing an $L^\infty$ bound on the state $u$. However, this introduces additional challenges in terms of the regularity of the resulting barrier functional. 

Using a similar process as in the previous section, we can determine \(\dot{B}\) explicitly, and reduce the invariance condition \(\dot{B} +\beta(B(u))\ge0\) for a class \(\mathcal{K}\) function \(\beta\) to ensuring:
\begin{align}
    u(t,a)f(u(t,a)) - u(t,b) f(u(t,b))\notag\\
    - F(u(t,a)) + F(u(t,b)) \le \beta(B(u(t))),\label{eq:boundary-cond-summaryinvariance}
\end{align}
For ease of calculation, we define the functions
\begin{itemize}[label={}, leftmargin=*]
    \item \(g(s,z) = (s-u^*)f(s) -(z-u^*)f(z) - F(s)+F(z)\),
    \item \(k(s,z) = sf(s) -zf(z) - F(s)+F(z)\),
\end{itemize}
which correlate to \eqref{eq:boundary-cond-summary} and \eqref{eq:boundary-cond-summaryinvariance} respectively, by setting \(s= u(t,a)\) and \(z=u(t,b)\).
\subsection{Intervals of convexity}
Our ultimate goal is to construct boundary controls via convex optimization programs that simultaneously enforce stability and invariance. To this end, we first identify the intervals over which $\dot V$ and $-\dot B$ are convex with respect to the boundary controls $u(t,a)$ and $u(t,b)$.
For $\dot V$, we compute the second derivatives of the left hand side in~\eqref{eq:boundary-cond-summary} with respect to $u(t,a)$ and $u(t,b)$. The intervals of convexity correspond to the regions where these second derivatives are nonnegative:
\begin{itemize}[label={}, leftmargin=*]
    \item \(\mathcal{C}_a\): interval within \([0,u_{\max}] \) where \(\dot{V}\) is convex in $u(t,a)$, i.e., $(u-u^*) f''(u) + f'(u) \ge 0$.
    \item \(\mathcal{C}_b\): interval within \([0,u_{\max}] \) where \(\dot{V}\) is convex in $u(t,b)$, i.e., $(u-u^*) f''(u) + f'(u) \le 0$.
\end{itemize}
It is clear that the intersection $\mathcal{C}_a\cap \mathcal{C}_b$ only consists of points where $(u-u^*) f''(u) + f'(u) = 0$, and $\mathcal{C}_a\cup \mathcal{C}_b = [0,u_{\max}]$. 
Repeating the same process, but for \(-\dot{B}\) by taking the second derivative of \eqref{eq:boundary-cond-summaryinvariance}, we get
\begin{itemize}[label={}, leftmargin=*]
    \item \(\mathcal{I}_a\): interval within \([0,u_{\max}] \) where $-\dot{B}$ is convex in $u(t,a)$, i.e., $u f''(u) + f'(u) \ge 0$.
    \item \(\mathcal{I}_b\): interval within \([0,u_{\max}] \) where $-\dot{B}$ is convex in $u(t,b)$, i.e., $u f''(u) + f'(u) \le 0$.
\end{itemize}

\begin{remark}\label{rem:covexintervals}
We can leverage the properties of the flux $f$ to determine the structure of the intervals of convexity of $\dot{V}$ and $-\dot{B}$, respectively. 
We can observe that \(0\in\mathcal{C}_a\) and \( 0\in\mathcal{I}_a\).
Indeed, since $f''(u)<0$ and $f'(u)>0$ on $[0,\hat u]$, both
$(u-u^*) f''(u) + f'(u)$ and $u f''(u) + f'(u)$ are positive in $u=0$.
Similarly, we can observe that \(u_{\max}\in \mathcal{C}_b\) and \( u_{\max}\in \mathcal{I}_b \). Indeed, \(f''(u) \) and \(f'(u)\) are both negative on \([\hat{u}, u_{\max}]\), therefore $(u-u^*) f''(u) + f'(u)$ and $u f''(u) + f'(u)$ are negative in $u=u_{\max}$.
Next, by following a similar argument, we can check what intervals include \(\hat{u}\). In doing so, we get \(\hat{u} \in \mathcal{C}_a \) if and only if \( \hat{u} \leq u^* \). Whilst, we get \( \hat{u} \in \mathcal{C}_b \) if and only if \( \hat{u} \geq u^* \). Finally, \(\hat{u} \notin \mathcal{I}_a\), regardless of \(u^*\), which by definition ensures \(\hat{u} \in \mathcal{I}_b\).
\end{remark}
Depending on the specific form of $f(u)$, some of the intervals $\mathcal{C}_a$, $\mathcal{C}_b$, $\mathcal{I}_a$, and $\mathcal{I}_b$ may consist of unions of multiple disjoint sets rather than a single contiguous interval. For the sake of simplicity in the calculations that follow, we will assume that each interval is contiguous, although all results and proofs remain valid in the general, non-contiguous case. Thus, from now on we assume they have the following form:
\begin{equation*}
\mathcal{C}_a=\left[0,u_1 \right], \ \mathcal{C}_b =\left[u_1, u_{\max} \right],\ u_1 \ \text{solves } u+\frac{f'(u)}{f''(u)} = u^*  ,
\end{equation*}
\begin{equation*}
\mathcal{I}_a=\left[0,u_2 \right], \  \mathcal{I}_b =\left[u_2, u_{\max}  \right],\ u_2 \ \text{solves } u +\frac{f'(u)}{f''(u)} = 0 .
\end{equation*}

\begin{remark}[Proving \(\mathcal{I}_a\cap \mathcal{C}_a = \mathcal{I}_a\) and \(\mathcal{I}_b\cap \mathcal{C}_b = \mathcal{C}_b\)]
\label{rem:Rmk1}
From the definitions of the convexity intervals of $\dot V$ and $-\dot B$, we show that $\mathcal{I}_a \subseteq \mathcal{C}_a$. Indeed, $u \in \mathcal{C}_a$ if and only if $(u - u^*) f''(u) + f'(u) \ge 0$, while $u \in \mathcal{I}_a$ if and only if $u f''(u) + f'(u) \ge 0$. Since $u^* \in [0, u_{\max}]$ and $f''(u) < 0$, the term $-u^* f''(u)$ is non-negative, which immediately implies $\mathcal{I}_a \subseteq \mathcal{C}_a$. By a similar argument, one can also verify that $\mathcal{C}_b \subseteq \mathcal{I}_b$.

Because the inequalities defining $\mathcal{C}_a$, $\mathcal{I}_a$, $\mathcal{C}_b$, and $\mathcal{I}_b$ are pointwise in $u$, these inclusion relations hold even if the sets consist of disjoint unions of intervals.
\end{remark}

Finally, we define the functions $C(t) = \alpha(V(u(t))) \ge 0$ and $D(t) = \beta(B(u(t))) \ge 0$ for some class $\mathcal{K}$ functions $\alpha$ and $\beta$. These will serve as the decay thresholds for stabilization and invariance, respectively.

\section{Single Input Controller Feasibility}
\label{section 3}

In this section, we characterize the conditions under which a single boundary control input is sufficient to guarantee both stability and invariance. Such an input may be applied either at the left boundary---by selecting a value in $\mathcal{C}_a \cap \mathcal{I}_a =\mathcal{I}_a $---or at the right boundary, by selecting a value in $\mathcal{C}_b \cap \mathcal{I}_b =\mathcal{C}_b$ (see Remark~\ref{rem:Rmk1}). 

At a given time $t\in [0,T]$, at the left boundary, we seek a control input 
$\omega_a(t) \in\mathcal{I}_a$ that ensures \(g(\omega_a(t), z) \leq -C(t)\) and \(k(\omega_a(t),  z) \leq D(t)\) for the right boundary trace $z=u(t,b)$. This leads to the convex optimization problem
\begin{equation}
\label{eq23}
\begin{aligned}
    \min_{\omega_a \in\mathcal{I}_a}  \big\{\omega_a^2: g\big(\omega_a,\, u(t,b)\big) \le -C(t), \\k\big(\omega_a,\, u(t,b)\big) \le D(t)\big\}.
\end{aligned}
\end{equation}
Analogously, at the right boundary we seek 
$\omega_b(t) \in \mathcal{C}_b$ satisfying
\(g(u(t,a), \omega_b(t)) \le -C(t)\) and \(k(u(t,a),   \omega_b(t)) \le D(t)\) for the left side boundary trace $s=u(t,a)$. This leads to the convex optimization problem
\begin{equation}
\label{eq24}
\begin{aligned}
    \min_{\omega_b \in \mathcal{C}_b} \big\{\omega_b^2 :g\big(u(t,a),\, \omega_b\big) \le -C(t),\\
     k\big(u(t,a),\, \omega_b\big) \le D(t)\big\}.
\end{aligned}
\end{equation}
We minimize the square of the input to turn this into a convex optimization problem, where the controller selects the minimal input so that the constraints are satisfied.

\subsection{Left-sided one-input feasibility}
In this section, we examine the feasibility conditions for implementing a single controller at the left boundary by analyzing the optimization problem \eqref{eq23}. Our goal is to find the minimum $\omega_a(t) \in \mathcal{I}_a$ such that both the stability and invariance inequalities hold simultaneously:
\begin{equation*}
g(\omega_a(t), z) \le -C(t), \quad k(\omega_a(t), z) \le D(t).
\end{equation*}

Let $s := \omega_a(t)$ and the right boundary value $z := u(t,b)$ at time $t$ be fixed. Then, the convex optimization problem~\eqref{eq23} can be recast as minimizing a quadratic objective function 
over the set
\begin{equation*}
\mathcal{S} = \{ s \in \mathcal{I}_a : g(s,z) \le -C(t), \ k(s,z) \le D(t) \}.
\end{equation*}
Since $0 \in \mathcal{I}_a$, we know that $s^* = 0$ is optimal if it satisfies the constraints. Assume instead that $0 \notin \mathcal{S}$, so $s^* \neq 0$. In this case, the minimizer cannot be interior to $\mathcal{S}$, and thus $s^*$ lies on the boundary of the feasible set. That is, either one or both constraints must be active:
\begin{equation*}
g(s,z) = -C(t) \quad \text{and/or} \quad k(s,z) = D(t).
\end{equation*}
We define the sets corresponding to each active constraint,
$$
\mathcal{S}_g = \{ s : g(s,z) = -C(t) \}, \quad
\mathcal{S}_k = \{ s : k(s,z) = D(t) \},
$$
and the set of feasible boundary solutions,
$$
\mathcal{U} = \{ s \in \mathcal{S}_g : k(s,z) \le D(t) \} \cup 
\{ s \in \mathcal{S}_k : g(s,z) \le -C(t) \}.
$$
If $\mathcal{U} = \emptyset$, then the problem is infeasible. If $0 \in \mathcal{U}$, then $s^* = 0$. Otherwise, the minimizer $s^* = \omega_a^*(t)$ lies on the boundary and satisfies one of the following:
\begin{equation}
\label{eq25}
g(\omega_a^*(t), z) = -C(t), \quad k(\omega_a^*(t), z) \le D(t),
\end{equation}
\begin{equation}
\label{eq26}
g(\omega_a^*(t), z) \le -C(t), \quad k(\omega_a^*(t), z) = D(t).
\end{equation}

This provides the feasibility constraints. Let $\bar{\mathcal{U}}$ denote the set of optimal solutions satisfying both \eqref{eq25} and \eqref{eq26}.

\begin{theorem}
There exist    \( \ \bar{u}, \ u^* \in [0, u_{\max}] \) and class \(\mathcal{K}\) functions \(\alpha, \ \beta \ \) such that  the convex optimization problem~\eqref{eq23} is feasible.
\end{theorem}
\begin{proof}
For ease of calculation, let us define \(p(s)=g(s,u(t,b)) \) and \(\ell(s)= k(s,u(t,b))\) for a fixed \(z=u(t,b)\):
\begin{itemize}[label={}, leftmargin=*]

    \item \(p(s) = (s-u^*)f(s)-(z - u^*)f(z)-F(s) + F(z)\),
    \item \(\ell(s) =  s f(s) - z f(z) - F(s) + F(z) \).
\end{itemize}
Trivially, $\omega_a(t)^2$ is minimized by $\omega_a(t) = 0$, so we always start by checking whether this value is admissible. Else, \(\omega_a^*(t)\) lies at the boundary of the constraint sets, either \eqref{eq25} or \eqref{eq26}. Starting with equation \eqref{eq25}, we investigate the existence of solutions to \(g(\omega_a^*(t), \hat{z}) = -C(t)\) for some \(\hat{z}\) such that $k(\omega_a^*(t), \hat{z}) \le D(t) \). 
Thus, we look at the stability conditions in terms of \(p(s)\), with derivative \(p'(s) = (s - u^*)f'(s) \).
 
\textbf{Case 1: \(u^* \leq u_2. \quad \)}For \(s\in \left[0, u^*  \right)\), we know \((s - u^*)\) is negative and \(f'(s) \) is positive. However, for
\(s\in  \left(u^*, u_2 \right] \), both factors are positive. Therefore, 
\(p(s)\) is strictly decreasing from \( \left[0, u^*  \right)\) and strictly increasing from \(\left(u^*, u_2  \right] \). 
Hence, $p(s)$ attains a minimum at $s = u^*$, with $p(u^*) = p_{\min}$.

    \begin{itemize}[label={}, leftmargin=*]
            \item \textbf{a)} If \(p(0) \leq -C(t) \),  the optimal \(s\) for \eqref{eq25} is \(\omega_a^*(t)=0\).
            \item \textbf{b)} If \(p(u^*) \leq -C(t)< p(0) \), an admissible solution to \eqref{eq25} exists, equal to the root \(s^*\) of \(p(s)=-C(t)\) with the most minimal norm.
            \item \textbf{c)} If \(p(u^*) > -C(t) \), that means \(p_{\min}\) is greater than \(-C(t)\). Thus, no solution exists for \eqref{eq25} and we move onto finding a solution to \eqref{eq26}.
        \end{itemize}

\textbf{Case 2: \(u^* \geq u_2, \quad \)}For \(s\in \left[0,  u_2 \right] \),  we know \((s - u^*)\) is negative and \(f'(s) \) is positive. This entails that \(p'(s)\) is always negative for \(s\in \left[0,u_2  \right] \), which allows us to conclude that \(p(s)\) is strictly decreasing from \( \left[0,u_2  \right] \). Consequently, $p(s)$ attains its minimum at $s = u_2$, with $p(u_2) = p_{\min}$.

    \begin{itemize}[label={}, leftmargin=*]
            \item \textbf{a)} If \(p(0) \leq -C(t) \), the optimal \(s\) for \eqref{eq25} is \(\omega_a^*(t)=0\).
            \item \textbf{b)} If \(p(u_2) \leq -C(t)< p(0) \), there is an admissible solution for \eqref{eq25}, equal to the root \(s^*\) of \(p(s)=-C(t)\) with the most minimal norm. 
            \item \textbf{c)} If \(p(u_2) > -C(t) \), that means \(p_{\min}\) is greater than \(-C(t)\). Thus,  no solution exists for \eqref{eq25} and we move onto finding a solution to \eqref{eq26}.
        \end{itemize}
    All admissible solutions for \eqref{eq25} are put into \(\mathcal{\bar{U}}\). If no solution exists, we move onto equation \eqref{eq26}.

For equation \eqref{eq26}, we investigate the existence of solutions to \(k(\omega_a^*(t),z) = D(t)\) for some \(\hat{z}\) such that \(g(\omega_a^*(t), z) \le -C(t)\). Therefore, we move on to \(\ell(s)\) with derivative \(\ell'(s) = s f'(s)\).

\textbf{Case 1:  \( \ell(0) \leq D(t). \quad \)} In this case, 
\(\omega_a(t)=0\) is feasible, and so 
the optimal solution for~\eqref{eq26} is \(\omega^*_a(t)=0\).
   
\textbf{Case 2: \(\ell(0) > D(t). \quad \)} Notice that \(\ell'(s) \geq 0\) on the interval \(\mathcal{I}_a = \left[0, u_2 \right]\), because \(f'(s)  > 0\) on  \(\left[0, u_2 \right]\). Thus \(\ell(s)\) is strictly increasing, such that \(\ell(0) = \ell_{\min}\) on \(\mathcal{I}_a\). 
 This implies that there is no solution to \eqref{eq26}.

All valid solutions are put into \(\mathcal{\bar{U}}\). Once we have the set of admissible solutions, we choose the solution with minimal square value.
\end{proof}

\subsection{Right-sided one-input feasibility}

In this section, we examine the feasibility requirements for implementing a single controller at the right boundary by analyzing equation \eqref{eq24}. Our goal is to find the minimum $\omega_b(t) \in \mathcal{C}_b$ such that both the stability and invariance inequalities hold simultaneously:
$$
g(s, \omega_b(t)) \le -C(t), \quad k(s, \omega_b(t)) \le D(t).
$$

Following a process analogous to the left-boundary case, we define the feasibility constraints as follows. If $z = u_1$ satisfies the constraints, then $z^* = u_1$. Otherwise, the minimizer $z^* = \omega_b^*(t)$ lies on the boundary and satisfies one of
\begin{equation}
\label{eq27}
g(s, \omega_b^*(t)) = -C(t), \quad k(s, \omega_b^*(t)) \le D(t),
\end{equation}
\begin{equation}
\label{eq28}
g(s, \omega_b^*(t)) \le -C(t), \quad k(s, \omega_b^*(t)) = D(t).
\end{equation}

Let $\bar{\mathcal{W}}$ denote the set of solutions satisfying both \eqref{eq27} and~\eqref{eq28}.

\begin{theorem}
   There exist \( \ \bar{u}, \ u^* \in [0, u_{\max}] \) and class \(\mathcal{K}\) functions \(\alpha, \ \beta \ \) such that the convex optimization problem \eqref{eq24} is feasible.
\end{theorem}
\begin{proof}

For ease of calculation, let us define 
\(q(z)=g(u(t,a),z)\) and \(\rho(z)=k(u(t,a),z)\) for a fixed \(s=u(t,a)\):
\begin{itemize}[label={}, leftmargin=*]
    \item \(q(z) = (s - u^*)f(s) - (z - u^*)f(z) - F(s) + F(z)\),
    \item \(\rho(z) = s f(s) - z f(z) - F(s) + F(z) \).
\end{itemize}
 Trivially, \( \omega_b(t)^2 \) is minimized by \(\omega_b(t)=u_1\), so we will always start by checking \(u_1\). Starting with equation \eqref{eq27}, we investigate the existence of solutions to \(g(s, \omega_b^*(t)) = -C(t)\), for some \(\hat{s}\) such that \(k(s, \omega_b^*(t)) \le D(t)\). Hence, we look at the stability conditions of \(q(z)\), with derivative \(    q'(z) = -(z - u^*)f'(z)  \). For case 1 we will assume \(u^*<\hat{u}\), but repeating the steps assuming \(u^*>\hat{u}\) grants identical results, except with \(u^* \) in the cases instead of \(\hat{u}\).

\textbf{Case 1:  \( u^* < \hat{u}, \quad\) } 
For \(z \in \left[u_1,u^* \right)\), both \(-(z-u^*)\)  and \(f'(z)\) are positive. But, for \(z\in  \left(u^*, \hat{u}  \right) \), \(-(z-u^*)\) is negative, whereas \(f'(z)\) is still positive. Finally, for \(z\in  \left(\hat{u}, u_{\max} \right] \), both \(-(z-u^*)\)  and \(f'(z)\) are negative. 
This tells us \(q(z)\) is strictly increasing from \( \left[0, u^*  \right)\), strictly decreasing from \(\left(u^*, \hat{u}  \right) \), and strictly increasing from \( \left[\hat{u}, u_{\max} \right) \), such that \(q(z) = -C(t)\) has at most two roots. So, it has 
a minimum at \(\hat{u}\), such that  \(q(\hat{u}) = q_{\min}\).  
    
            \begin{itemize}[label={}, leftmargin=*]
                \item \textbf{a)} If \(q(u_1) \leq -C(t) \), the optimal \(z\) for \eqref{eq27} is \( \omega_b^*(t)=u_1\).
                \item \textbf{b)} If \(q(\hat{u}) \leq -C(t)< q(u_1) \), an admissible solution exists for \eqref{eq27}, which is the  root \(z\) of \(q(z) =-C(t)\) with the most minimal norm. 
                \item \textbf{c)} If \(q(\hat{u}) > -C(t) \), that means \(q_{\min}\) is greater than \(-C(t)\). Thus, no solution exists for \eqref{eq27} and we move onto finding a solution to \eqref{eq28}.
            \end{itemize}

    
\textbf{Case 2:  \( u^* = \hat{u}, \quad \) } For \(z\in \left[0, \hat{u}  \right)\), both \(-(z-u^*)\) and \(f'(z)\) are  positive, while for \(z\in  \left(\hat{u}, u_{\max} \right] \), both \(-(z-u^*)\) and \(f'(z)\) are negative. Therefore, \(q'(z)\) is always non-negative and 
\(q(z)\) is strictly increasing. 
This means \(q(z)\) has a maximum at \(u_{\max}\) and a minimum at \(u_1\), such that \(q(u_1) = q_{\min}\).  
        \begin{itemize}[label={}, leftmargin=*]
            \item \textbf{a)} If \(q(u_1) \leq -C(t) \), the optimal \(z\) for \eqref{eq27} is \( \omega_b^*(t)=u_1\).
 
            \item \textbf{b)} If \(q(u_1) > -C(t) \), that means \(q_{\min}\) is greater than \(-C(t)\). Thus, no solution exists for \eqref{eq27} and we move onto finding a solution to \eqref{eq28}.
        \end{itemize}

    All admissible solutions for \eqref{eq27} are put into \(\mathcal{\bar{W}}\). If no solution exists, we move onto equation \eqref{eq28}.

For equation \eqref{eq28}, we investigate the existence of solutions to \(k(s,\omega_b^*(t)) = D(t)\), for some \(\hat{s}\) such that \(g(s, \omega_b^*(t)) \le -C(t)\). Therefore we move on to \(\rho (z)\), with derivative \(\rho'(z) = -zf'(z)\).

 \textbf{Case 1:  \( u_1 \leq \hat{u}, \ \quad \)}
    \(\rho'(z)\) is always negative from \(\left[u_1, \hat{u}\right)\), and always positive from \(\left(\hat{u}, u_{\max} \right]\). Thus, \(\rho(z)\) is decreasing from \(u_1\) to \(\hat{u}\) and increasing from \(\hat{u}\) to \(u_{\max}\). 
    That means \(\rho(z)\) has a minimum at \(\hat{u}\) and local maximums at both \(u_1\) and \(u_{\max}\), such that \(\rho(\hat{u}) = \rho_{\min}\). 
            \begin{itemize}[label={}, leftmargin=*]
            \item \textbf{a)} If \(\rho(u_1) \leq D(t)\), the optimal \(z\) for \eqref{eq28} is \( \omega_b^*(t)=u_1\).
            \item \textbf{b)} If \(  \rho(u_1) > D(t) \geq \rho(\hat{u}) \), it has an admissible solution to \eqref{eq28}, which is the root \(z\) of \(\rho(z)= D(t)\).
 
            \item \textbf{c)} If \( \rho(\hat{u}) > D(t) \), then it means even the smallest \(\rho(z)\) is too large and no solution exists for \eqref{eq28}.
        \end{itemize}
    
\textbf{Case 2:  \( u_1 \geq \hat{u}, \quad \)}
\(\rho'(z) \) is always positive from \([u_1,u_{\max}]\), such that it is strictly increasing from \([u_1,u_{\max}]\). Since this entails that \(\rho_{\min}=\rho(u_1)\), 
we know \(z\) must be \( \omega_b^*(t)=u_1\), or else no solution exists or \eqref{eq28}. 

All admissible solutions are put into \(\mathcal{\bar{W}}\). Once we have the set of admissible solutions, we choose the solution with minimal square value.
\end{proof}
\section{Simulations}
\label{section 4}
\begin{figure*}[!t]

\includegraphics[trim={0 0 0 0},clip,width= 1\textwidth]{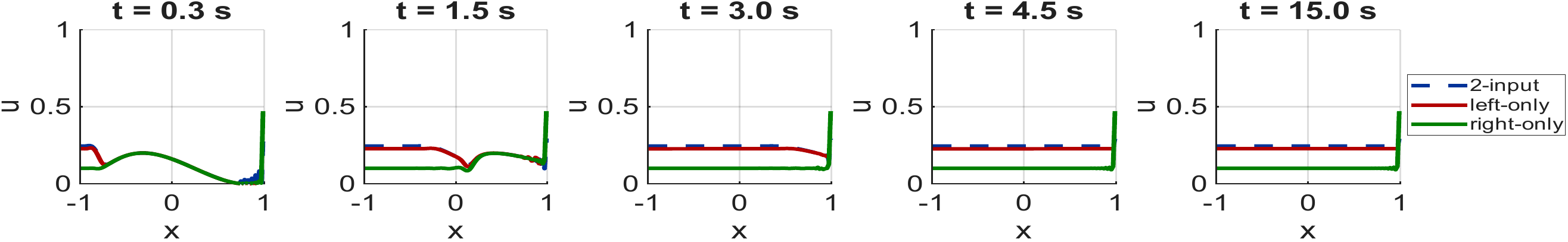} %
\caption{ For Example \ref{example1}, this figure shows that all three controllers attempt to stabilize the initial signal towards \(u^*\). This shows that single-input controllers can be as valid as the two-input controller. Each get close to \(u^* = 1/3\), without going past \(\bar{u} =1/4.\)} 
\label{fig:allsucceed}

\end{figure*}



We now illustrate the feasibility of ensuring both stability and invariance from a single boundary through some numerical examples. 
Our first goal is to demonstrate that these controllers successfully enforce both properties.  

We simulate the IBVP~\eqref{eq1}--\eqref{eq2} on the domain $[a,b] = [-1,1]$ with $u_{\max} = 1$, using a sinusoidal initial density
$u_0(x) = 0.1 - 0.1 \sin(\pi x)$.  The conservation law is solved using a Finite Volume Method, with the numerical flux modeled as Rusanov, using forward Euler for the time integration. Next, we intentionally choose $u^*$—used to define the Lyapunov functional in \eqref{eq:V-summary}—larger than $\bar u$—used in \eqref{eq:B-summary}—so that the system must balance the objective of reaching a desired state outside the safe set.  
In each simulation, control updates are computed every $0.015$ seconds over a total duration of 15 seconds.\\
\begin{example}
 For the first simulation we set \(u^*=1/3\) and \(\bar{u}=1/4\), and use the quadratic flux function from the Greenshield's model:
 \begin{equation}
     f(u) = u\left(1-\frac{u}{u_{\max}}\right) =u(1-u)     \label{fluxfinal}
 \end{equation}
Since this is the same model that the prior two-input paper~\cite{chiri2025boundary} proved feasibility for, we can directly compare our controller to their two-input controller. 


\begin{figure}[H]
\begin{center}
\includegraphics[width=8.4cm,height=3.5cm]{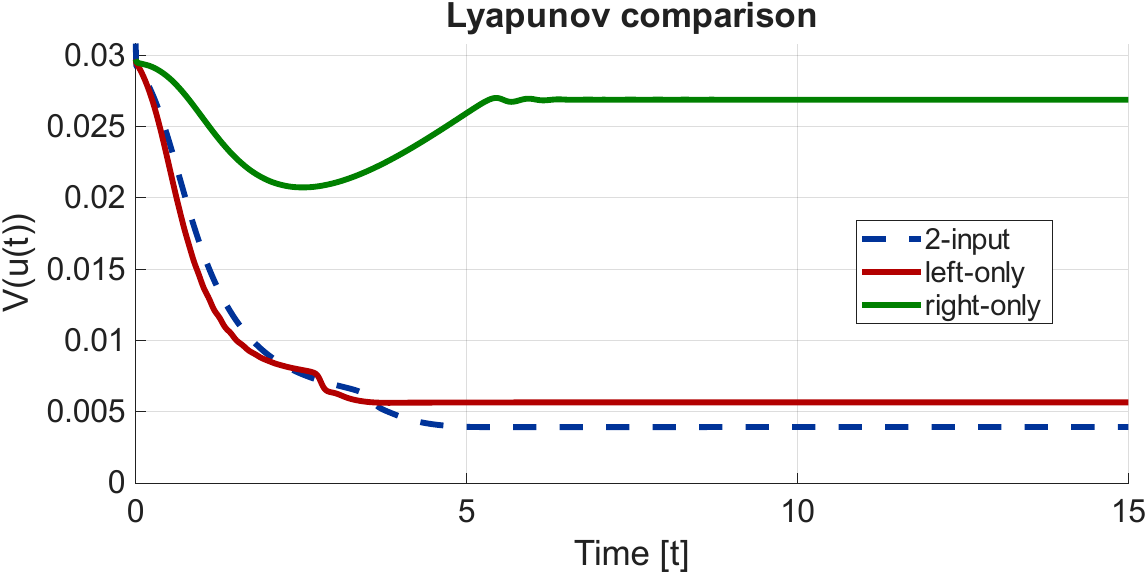} 
\caption{For Example \ref{example1}, this figure shows the change in \(V(t)\) over time, as the density gets closer to \(u^*\). Notably, the two-input system drives $V(t)$ closest to zero, indicating a more optimal performance compared to the other controllers. No controller reaches zero, as none can drive the density \(u\) to \(u^*\) without exiting the safe set of \(\bar{u}\), for this specific example.} 
\label{fig:allsucceedlyapunov}
\end{center}
\end{figure}

\begin{figure}[H]
\begin{center}
\includegraphics[width=8.4cm,height=3.5cm]{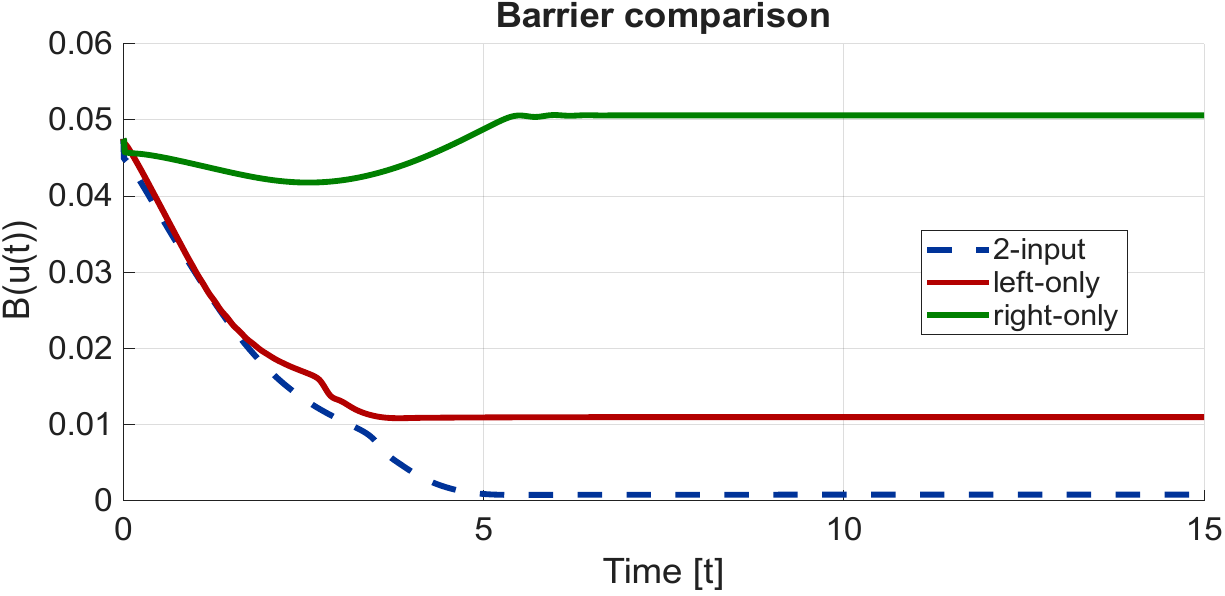} 
\caption{ For Example \ref{example1}, this figure depicts the change in \(B(t)\) over time for the first simulation. The closer \(u\) gets to \(\bar{u}\) the closer \(B(t)\) gets to zero. Here, none of the controllers leave the safe set, by never sinking below zero.} 
\label{fig:allsucceedbarrier}
\end{center}
\end{figure}

\begin{figure}[H]  
\begin{center}
\includegraphics[width=8.4cm]{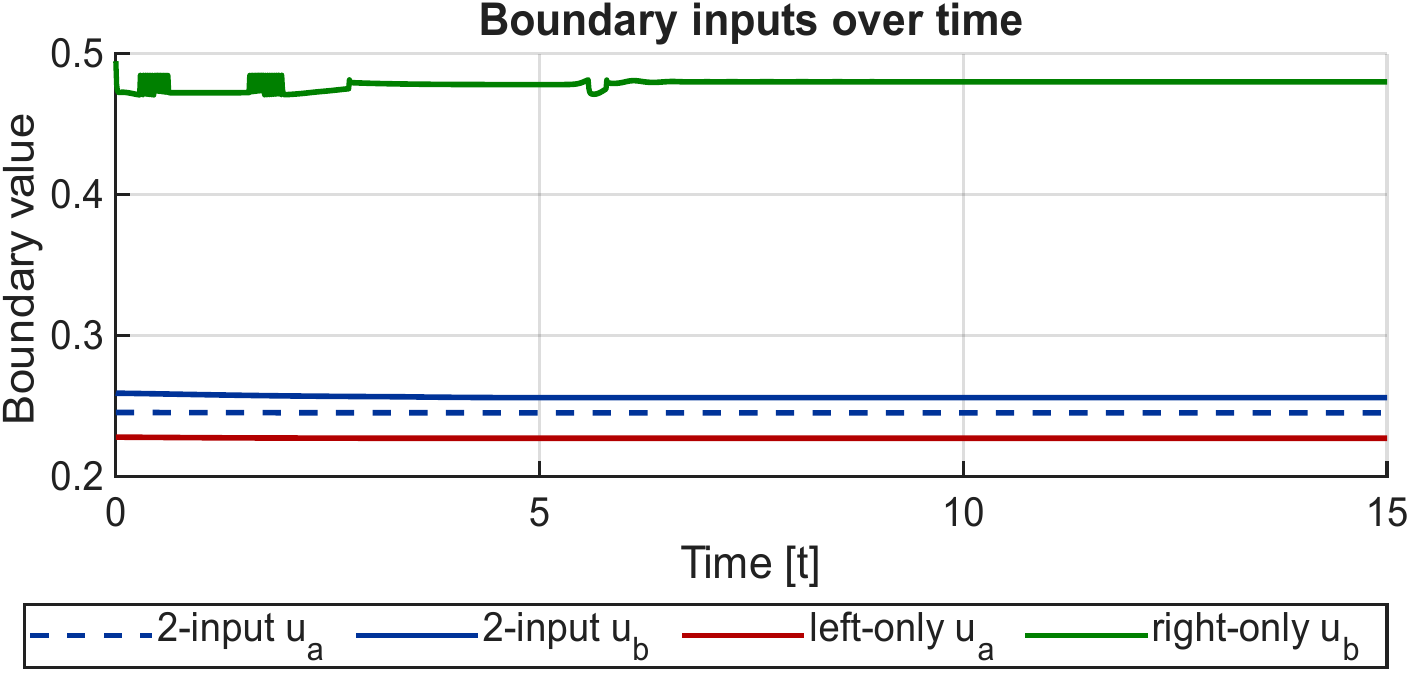} 
\caption{For Example \ref{example1}, this figure displays the boundary inputs taken by each controller over the 15 seconds. } 
\label{fig:inputs}
\end{center}
\end{figure}

In this setting, \(u_1 =\frac{2u^*+u_{\max}}{4} = 5/12 \approx0.417\) and \(u_2=\frac{u_{\max}}{4} = 1/4\). This gives us  \(\mathcal{C}_a = [0, 0.417], \ \mathcal{I}_a =[0, 0.25], \ \mathcal{C}_b = [0.417, 1], \text{ and }\mathcal{I}_b =[0.25, 1]\). Thus, the right-only controller fails to stabilize close to the value of $\bar{u}$, since the right-only controller can only pull values from \(\mathcal{C}_b\in [0.417,1]\). In Figure \ref{fig:allsucceedlyapunov}, the two-input case is able to get the closest to an optimal solution, because by having two inputs, it has a larger region of feasibility. Finally, whilst the left-side controller stabilized the signal at a lower input in Figure \ref{fig:inputs}, it comes at the cost of not stabilizing the signal as well as the aforementioned two-input case.\\
\label{example1}
\end{example}
\begin{example}
For the second simulation, we set \(u^*=1/4\) and \(\bar{u}=1/5\). Then, to validate the applicability of the proposed method to more general flux functions, we define the flux as a sextic polynomial: 
\begin{equation}
    f(u) = u\left(1-\frac{u}{u_{\max}}\right)(u^4+2u^3+3u^2+4u+5)
    \label{polyflux}
\end{equation}
which satisfies the properties required in Section~\ref{subsec:IBVP}.\\
\label{example2}
\end{example}

\begin{figure*}
\begin{center}
\includegraphics[trim={0 0cm 0cm 0cm},clip,width=1\textwidth, height =3.3cm]{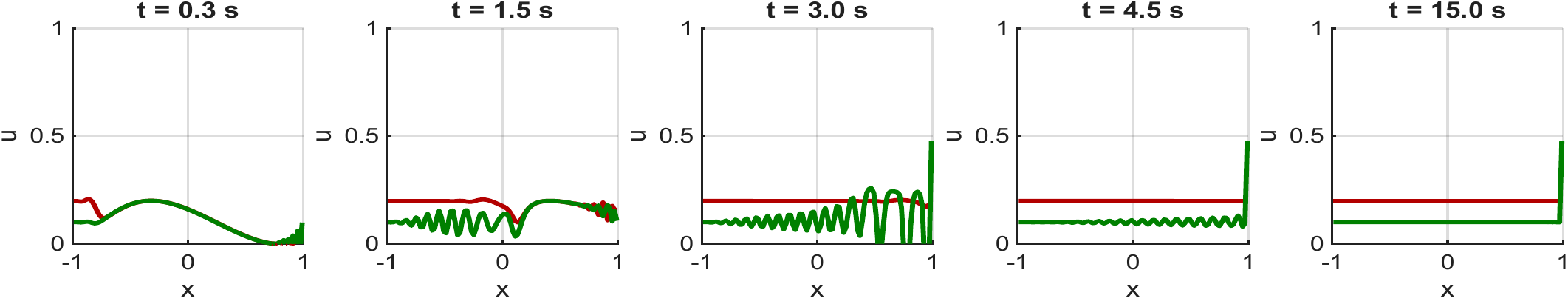}
\caption{For Example \ref{example2}, this figure demonstrates both the right-side and left-side case successfully stabilize the initial signal, with the left-side controller getting much closer to \(u^*=1/4\). 
} 
\label{fig:sextic}
\end{center}
\end{figure*}
\begin{example}
Finally, we choose a logarithmic flux function, based on the modified Greenberg Model (\cite{fluxtypes}). Here, we set \(u^*=1/4\), \(\bar{u}=1/5\) and \(\varepsilon = 0.1\), and define the flux function by
\begin{equation}
    f(u) = u\log\left(\frac{u_{\max}+\varepsilon}{u+\varepsilon}\right)=u\log\left(\frac{1.1}{u+0.1}\right)
    \label{logflux}
\end{equation}
\label{example3}
\end{example}

\begin{figure*}
\begin{center}
\includegraphics[trim={0 0cm 0cm 0},clip,width=1\textwidth, height =3.3cm]{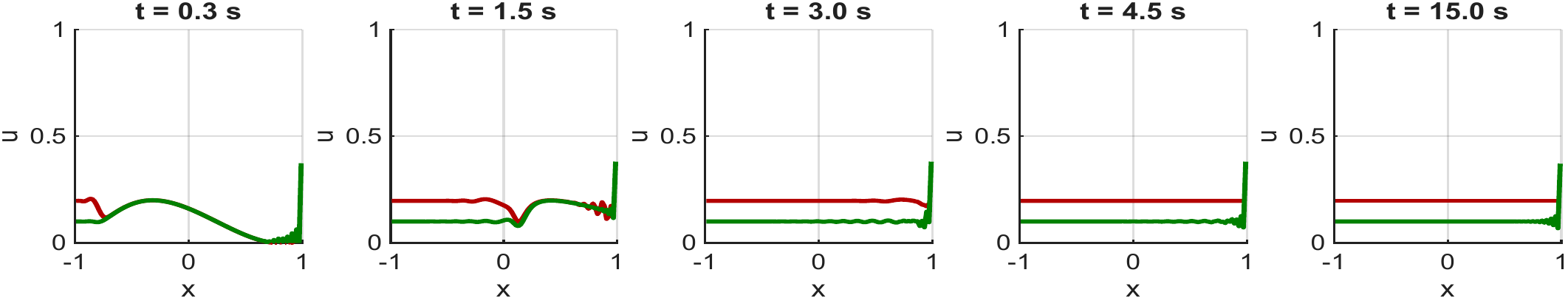}
\caption{For Example \ref{example3}, the figure demonstrates both the right-side and left-side cases successfully stabilize the initial signal, with the left-side control again getting the density closer towards \(u^* =1/4\). 
} 
\label{fig:log2}
\end{center}
\end{figure*}


Examples \ref{example2} and \ref{example3} help demonstrate the wide variety of flux functions we can apply this methodology to. Example~\ref{example2} shows the controllers capacity to stabilize higher order polynomials, while Example~\ref{example3} shows the capacity to stabilize non polynomial functions. Additionally, neither controller ever had \(B(t)\) sink below zero, which shows that our attempts to achieve stability, never negated the invariance for both flux functions. By succeeding for flux functions~\eqref{polyflux} and \eqref{logflux}, the controllers have demonstrated their applicability for any LWR model, as individual models differ by how they define their flux.

Whilst the goal of this paper is feasibility, it is important to note due to the nature of system, in which the traffic inflows from the left-side and outflows through the right-side, the right-side only controller can face unique challenges that can cause it to perform sub-optimally.

\section{Conclusion}

This paper presents a convex-optimization method to derive optimal boundary controllers of a PDE system describing traffic flow. Here, we used a convex optimization approach to find the conditions required to optimally ensure stability and invariance of controlled trajectories. Additionally, all results in this paper can be extended to the cases of more complex flux functions that have non-contiguous intervals of convexity. In the future, we envision to generalize this method to the case of junctions and traffic models on networks.


\bibliography{ifacconf}             
                                                   







\end{document}